\documentclass[11pt]{amsart}
\usepackage{amsmath,amsthm,amsfonts,amscd,flafter,epsf,amssymb,manfnt,epsdice}
\usepackage{epsfig}
\usepackage{amsfonts}
\usepackage{amssymb}
\usepackage{amsmath}
\usepackage{amsthm}
\usepackage{latexsym}
\usepackage{amscd}
\usepackage{flafter}
\usepackage{epsf}
\input{diagrams}

\newcommand{\hf}{\mathfrak{h}}

\newcommand{\taub}{\Xi}

\newcommand{\ov}{\widehat}

\newcommand{\Sig}{\Sigma}
\newcommand{\ra}{\rightarrow}

\newcommand{\B}{\mathbb{B}}

\newcommand{\rank}{\mathrm{rk}}

\newcommand\Dual{\mathcal D}
\newcommand\Duality\Dual

\newcommand{\F}{\Z/2\Z}

\newcommand\RelSpinC{\underline{\SpinC}}
\newcommand\relspinc{s}

\newcommand\x{\mathbf x}

\newcommand\z{\mathbf z}

\newcommand\y{\mathbf y}

\newcommand\ModSphere{\ModFlow\left({\mathbb S}\longrightarrow
\Sym^{g-1}(\Sigma_{1})\times \Sym^2(\Sigma_{2})\right)}
\newcommand\ModSpheres\ModSphere

\newcommand\UnparModSp{\widehat \ModSp}
\newcommand\UnparModFlow\UnparModSp

\newcommand\spin{\mathfrak s}

\newcommand{\spinc}{\mathfrak s}

\newcommand{\spinct}{\mathfrak t}

\newcommand\ModMaps{\mathcal M}
\newcommand\ModSp\ModMaps

\newcommand\alphas{\mbox{\boldmath$\alpha$}}

\newcommand\betas{\mbox{\boldmath$\beta$}}

\newcommand\spincrel\relspinc

\newtheorem{thm}{Theorem}[section]

\newtheorem{cor}[thm]{Corollary}

\def\endproof{\relax\ifmmode\expandafter\endproofmath\else
  \unskip\nobreak\hfil\penalty50\hskip.75em\hbox{}\nobreak\hfil\bull
  {\parfillskip=0pt \finalhyphendemerits=0 \bigbreak}\fi}
\def\endproofmath$${\eqno\bull$$\bigbreak}
\def\bull{\vbox{\hrule\hbox{\vrule\kern3pt\vbox{\kern6pt}\kern3pt\vrule}\hrule}}

\newcommand{\Z}{\mathbb{Z}}

\newcommand{\ModSWfour}{\mathcal{M}}
\newcommand{\ModFlow}{\ModSWfour}

\newcommand{\SpinC}{{\mathrm{Spin}}^c}

\newcommand\abuts\Rightarrow
\newcommand\Sym{\mathrm{Sym}}

\newcommand{\Hbb}{{\mathbb{H}}}

\newcommand{\Fbb}{\mathbb{F}}

\newcommand{\lra}{\longrightarrow}

\begin{document}

\title[Knots with a surgery that has simple Floer homology]{ Knots which admit a surgery with simple knot Floer homology groups}%
\author{Eaman Eftekhary}%
\address{School of Mathematics, Institute for Research in Fundamental Sciences (IPM),
P. O. Box 19395-5746, Tehran, Iran}%
\email{eaman@ipm.ir}
\keywords{simple knot Floer homology, L-space surgery}%

\maketitle

\begin{abstract}
We show that if a positive integral surgery on a knot $K$ inside a homology sphere $X$ results in an induced knot $K_n\subset X_n(K)=Y$ which has
simple Floer homology then we should have $n\geq 2g(K)$. Moreover, for $X=S^3$ the three-manifold $Y$ is a $L$-space, and the Heegaard Floer homology groups of $K$ are determined by its Alexander polynomial.
\end{abstract}
\section{Introduction}
The importance of studying knots inside rational homology spheres which have simple knot Floer homology
came up in the study of Berge conjecture using techniques from Heegaard Floer homology by 
Hedden \cite{Matt} and Rasmussen \cite{Ras2}. By definition, a knot $K$ inside a rational 
homology sphere $X$ has simple knot Floer homology if the rank of $\ov{\mathrm{HFK}}(X,K)$ 
is equal to the rank of $\ov{\mathrm{HF}}(X)$ (see \cite{OS-3mfld,Ras1,OS-knot} for the background 
on Heegaard Floer homology and knot Floer homology). 
Berge conjecture about the knots in $S^3$ which
admit a Lens space surgery may almost be reduced to showing that a knot inside a Lens space 
with simple knot Floer homology is simple.\\
For knots inside arbitrary rational homology spheres, it is not clear 
what the topological implications of having simple knot Floer homology are.
With more restriction on the ambient three manifold however, certain conclusions may be made.
In particular, if the ambient manifold $X$ is an integer homology sphere, the author has shown \cite{Ef-rank}
that the only knot with simple knot Floer homology is the trivial knot. In this paper, we prove two more theorems
in this direction. The first theorem is about the knots obtained by small 
surgery from knots inside homology spheres. If $K\subset X$ is a null-homologous knot inside a three-manifold 
$X$, we may remove a tubular neighborhood of $K$ and glue it back in a different way so that the resulting 
manifold $X_n(K)$ is the three-manifold obtained from $X$ by $n$-surgery on $K$. In this situation 
the core of the new solid torus will determine a knot $K_n\subset X_n(K)$. We show:
\begin{thm}
Suppose that $K\subset X$ is a knot inside a homology sphere $X$ of Seifert genus $g(K)$.
Suppose that $0<n<2g(K)$ is a given integer and $K_n\subset X_n(K)$ is the knot obtained from $K$ by
$n$-surgery. Then $K_n$ can not have simple knot Floer homology, i.e.
\begin{equation}
\rank\big(\ov{\mathrm{HFK}}(X_n(K),K_n)\big)>\rank\big(\ov{\mathrm{HF}}(X_n(K))\big)
\end{equation}
\end{thm} 
When the surgery coefficient $n$ is greater than or equal to $2g(K)$, there is more freedom for choosing 
$K$ so that $K_n$ has simple knot Floer homology. 
In particular, a necessary and sufficient condition may be given when $X$ is a  $L$-space.
In order to state the precise theorem, for a knot $K$ inside the homology sphere $L$-space $X$ 
let $\B=\B(K)$ denote the vector space $\ov{\mathrm{HFK}}(X,K;\Z/2\Z)$ and let $d_\B:\B\ra \B$ 
denote the differential obtained by counting the disks passing through the second marked point in 
a doubly pointed Heegaard diagram associated with the pair $(X,K)$. The homology group $\mathrm{H}_*(\B(K),d_\B)$
is then equal to $\ov{\mathrm{HF}}(X;\Z/2\Z)=\Z/2\Z$. Let $d(X)$ denote the homological degree
of the generator of this later vector space.
\begin{thm}\label{thm:main}
Let $K\subset X$ be a knot in a homology sphere $L$-space $X$ of Seifert genus $g(K)$.
If for an integer $n\geq 2g(K)$ the knot $K_n\subset X_n(K)$ has simple knot Floer homology, 
$X_n(K)$ is a $L$-space and
 there is an increasing sequence of integers
$$-g(K)=n_{-k}<n_{1-k}<...<n_k=g(K)$$
with $n_i=-n_{-i}$ for which the following is true. For $i\in \Z$ with $|i|\leq g(K)$ define
$$\delta_i=\begin{cases}
0+d(Y) &\text{if }i=g(K)\\
\delta_{i+1}+2(n_{i+1}-n_i)-1 &\text{if } i<g(K),\ \&\ g(K)-i\equiv 1\ (\text{mod }2)\\
\delta_{i+1}+1 &\text{if } i<g(K),\ \&\ g(K)-i\equiv 0\ (\text{mod }2).\\
\end{cases} $$
In this situation, $\ov{\mathrm{HFK}}(K,\spinc)=0$ unless $\spinc=n_i$ for some $-k\leq i\leq k$, in which case
$\ov{\mathrm{HFK}}(K,\spinc)=\Fbb$ and it is supported entirely in homological degree $\delta_i$.
 If the generator of $\ov{\mathrm{HFK}}(K,n_i)$
is denoted by $\x_i$, the filtered chain complex $(\B(K),d_\B)$ may be described (up to quasi-isomorphism) as
$\B=\langle \x_{-k},...,\x_k\rangle$, where the differential is given by
$$d_\B(\x_i)=\begin{cases}
0\ &\text{if }i=k,\ \text{or } 0<k-i \text{ is odd,}\\
\x_{i+1}\ &\text{otherwise.}
\end{cases}$$
Moreover, if for a knot $K$ inside a homology sphere $L$-space $X$ the filtered chain complex
$(\B(K),d_\B)$ has the above form, for any integer $n\geq 2g(K)$, $X_n(K)$ is a 
$L$-space and $K_n\subset X_n(K)$ has simple knot Floer homology. 
\end{thm}
This gives a complete classification of knots with simple knot Floer homology 
which are obtained by surgery on a knot $K\subset S^3$ in terms of hat Heegaard Floer 
homology of the knot. In particular, if $n$-surgery on a
knot $K\subset S^3$ has simple knot Floer homology, $n\geq 2g(K)$ and 
all the coefficients of the symmetrized Alexander polynomial associated with $K$ are equal to $1$.\\
We hope that the techniques used here are useful in the study of general knots with simple knot Floer 
homology, although understanding such knots in its full generality requires significant breakthroughs at 
this point. In particular, Berge conjecture is wide open from this prospective.\\
\\
{\bf{Acknowledgement.}} I would like to thank Matt Hedden for bringing up this problem and helpful discussions during a visit to MSRI. I would also like to thank MSRI for providing us with an opportunity for such interactions.
\section{Knot Floer homology background}
\subsection{Relative $\SpinC$ structures and rationally null-homologous knots}
Let $X$ be a homology sphere and $K$ be a knot inside $X$.
Consider a tubular neighborhood nd$(K)$ of $K$ and let $T$ be the torus boundary of this neighborhood. Let
$\mu\subset T$ be a meridian of $K$, i.e. $\mu$ bounds a disk in nd$(K)$, and let $\lambda\subset T$ be a
zero framed longitude for $K$, which is a curve that is isotopic to $K$ in nd$(K)$ and bounds a Seifert
surface $S$ for $K$ in $X-\mathrm{nd}(K)$.
We may assume that $\lambda$ and
$\mu$ intersect each other in a single transverse point. Having fixed these two curves, by $(p,q)$-surgery on
$K$ we mean removing nd$(K)$ and replacing for it a solid torus so that the simple closed curve $p\mu+q\lambda$
bounds a disk in the new solid torus. Denote the resulting three-manifold by $X_{p/q}(K)$. The core of the
new solid torus would be an image of $S^1$ in $X_{p/q}(K)$ which will be denoted by $K_{p/q}\subset X_{p/q}(K)$.
Let $\Hbb_{p/q}(K)$ be the Heegaard Floer homology group $\ov{\mathrm{HFK}}(X_{p/q}(K),K_{p/q};\Z/2\Z)$.
Relative $\SpinC$ structures on $X-\mathrm{nd}(K)$ which reduce to the translation invariant vector field
on the boundary form an affine space $\RelSpinC(X,K)$ over $\mathrm{H}^2(X,K;\Z)$, and clearly we will have
$\RelSpinC(X,K)=\RelSpinC(X_{p/q}(K),K_{p/q})$ in a natural way. The group $\Hbb_{p/q}(K)$ is decomposed
into subgroups associated with relative $\SpinC$ structures:
\begin{equation}
\Hbb_{p/q}(K)=\bigoplus_{\spinc\in \RelSpinC(X,K)}\Hbb_{p/q}(K,\spinc).
\end{equation}
There is a natural involution
\begin{equation}
J:\RelSpinC(X,K)\lra \RelSpinC(X,K)
\end{equation}
which takes a $\SpinC$ class $\spinc$ represented by a nowhere vanishing vector field $V$
on $X-\mathrm{nd}(K)$, to the $\SpinC$ class $J(\spinc)$ represented by $-V$. The difference
$\spinc-J(\spinc)\in \mathrm{H}^2(X,K;\Z)$ is usually denoted by $c_1(\spinc)$.
There is a symmetry in knot Floer homology 
which may be described by the following formula
\begin{equation}\label{eq:duality}
\ov{\mathrm{HFK}}(X,K,\spinc)\simeq \ov{\mathrm{HFK}}(X,K,J(\spinc)+\mathrm{PD}[\mu]).
\end{equation}
Since $X$ is a homology sphere, the cohomology group $\mathrm{H}^2(X,K;\Z)$ is generated
by the class $\mathrm{PD}[\mu]$ and we may thus naturally identify $\mathrm{H}^2(X,K;\Z)$
with $\Z$. We then have a map
\begin{equation}
\begin{split}
&\hf:\RelSpinC(X,K)\lra \Z=\mathrm{H}^2(X,K;\Z)\\
&\hf(\spinc):=\frac{c_1(\spinc)-\mathrm{PD}[\mu]}{2}
\end{split}
\end{equation}
which satisfies $\hf(\spinc)=-\hf(J(\spinc)+\mathrm{PD}[\mu])$ for all relative $\SpinC$ classes
$\spinc\in\RelSpinC(X,K)$. Using this map $\RelSpinC(X,K)$ may also be identified with $\Z$ in a
natural way.\\
The following theorem of Ozsv\'ath and Szab\'o \cite{OS-genus}, generalized by Ni \cite{Ni}
to rationally null-homologous knots,  allows us compute the genus of a knot $K\subset X$,
using Heegaard Floer homology:
\begin{thm}
If $K\subset X$ is a knot inside a homology sphere $X$ as above, the Seifert genus $g(K)$
of $K$ may be computed from
\begin{equation}
g(K)=\max \big\{\spinc\in\RelSpinC(X,K)=\Z\ \big|\ \ov{\mathrm{HFK}}(X,K;\spinc)\neq 0\big\}.
\end{equation}
\end{thm}
\subsection{Surgery formulas}
Suppose that $(\Sig,\alphas,\betas;u,v)$ is a Heegaard diagram for the knot $K$, such that
$\betas=\betas_0\cup\{\mu=\beta_g\}$, $(\Sig,\alphas,\betas_0)$ is a Heegaard diagram for
$X-\mathrm{nd}(K)$, while $\mu=\beta_g$ represents the meridian of $K$ and the two marked points
$u$ and $v$ are placed on the two sides of $\beta_g$. Think of the vector space
$\B=\ov{\mathrm{HFK}}(X,K;\Z/2\Z)=\Hbb_\infty(K)$ as a vector space computed as
$\B=\ov{\mathrm{HF}}(\Sig,\alphas,\betas;u,v)$.
Letting holomorphic disks pass through the marked point
$v$ in the Heegaard diagram gives a map $d_\B:\B\ra \B$, which is a filtered differential on the
filtered vector space $\B$, with the
filtration induced by relative $\SpinC$ structures.  The homology of the complex $(\B,d_\B)$ gives
$\ov{\mathrm{HF}}(X;\Z/2\Z)=H_*(\B,d_\B)$.
For a relative $\SpinC$ class $\spinc\in\RelSpinC(X,K)=\Z$ we set
$$\B\{\geq \spinc\}=\bigoplus_{\substack{\spinct\in \RelSpinC(X,K)=\Z\\ \spinct\geq \spinc}}\B(\spinct).$$
Then the subspace $\B\{\geq \spinc\}$ of $\B$ is mapped to itself by the differential $d_\B$ of $\B$.
Furthermore, let $\imath_\spinc:\B\{\geq \spinc\}\ra \B$ be the inclusion map, and
denote the homology of $(\B,d_\B)$ by $\Hbb$ and the homology of $\B\{\geq \spinc\}$ by $\Hbb\{\geq \spinc\}$.\\


The following theorem which gives an explicit formula for the groups $\Hbb_{n}(K,\spinc)$
is proved in \cite{Ef-rank}:
\begin{thm}\label{thm:knot-surgery-formula}
Suppose that $K\subset X$  is a  knot inside a  homology sphere $X$.
 With the above notation fixed, the  group $\ov{\mathrm{HFK}}(X_{n}(K),K_{n},\spinc;\Z/2\Z)$
 may be computed as the homology of the complex
$$C_{n}(\spinc)=\B\{\geq \spinc\}\oplus
\B\{\geq n+1-\spinc\}\oplus \B,$$
which is equipped with a differential $d_{n}:C_{n}(\spinc)\ra C_{n}(\spinc)$ defined by
$$d_{n}(\x,\y,\z)=(d_\B(\x),d_\B(\y),d_\B(\z)+\imath_{\spinc}(\x)+\imath_{n+1-\spinc}(\y)).$$
\end{thm}
Note that instead of the complex $C_n(\spinc)$ we may consider a complex 
$$\tilde{C}_n(\spinc)=\Hbb\{\geq \spinc\}\oplus \Hbb\{\geq n+1-\spinc\}\oplus \Hbb$$
which is equipped with a differential $\tilde{d}_n:\tilde{C}_n(\spinc)\ra \tilde{C}_n(\spinc)$ defined
by
$$\tilde{d}_{n}(\x,\y,\z)=(0,0,(\imath_{\spinc})_*(\x)+(\imath_{n+1-\spinc})_*(\y)).$$
The homology groups $\mathrm{H}_*(\tilde{C}_n(\spinc),\tilde{d}_n)$ and 
$\mathrm{H}_*({C}_n(\spinc),{d}_n)$ are then isomorphic.\\
There is a similar formula for the Heegaard Floer homology groups associated with the three-manifold
$X_n(K)$ which is due to Ozsv\'ath and Szab\'o. We have slightly modified the statement of their theorem
from \cite{OS-surgery} so that it looks more compatible with the notation of theorem~\ref{thm:knot-surgery-formula}.
To state the theorem, let $\pi_\spinc:\B\{\geq \spinc\}\ra \B\{\spinc\}=\frac{\B\{\geq \spinc\}}{\B\{>\spinc\}}$ 
be the projection map,
$\Xi_\spinc:\B\{\spinc\}\ra \B\{-\spinc\}$ be the duality isomorphism, 
and $\jmath_\spinc:\B\{\spinc\}\ra \B$ be the
inclusion map. Associated with any $\SpinC$ class $[\spinc]\in\SpinC(X_n(K))=\Z/n\Z$ we may construct
a chain complex
$$F_n[\spinc]=\bigoplus_{\spinc\in[\spinc]\subset \Z}C_n(\spinc)$$
which is equipped with a differential $f_n:F_n[\spinc]\ra F_n[\spinc]$ defined as follows on a generator
$(\x,\y,\z)\in C_n(\spinc)$:
\begin{equation}
\begin{split}
&f_n(\x,\y,\z)=(\alphas_\spinct)_{\spinct\in[\spinc]\subset \Z},\ \ \ \alphas_\spinct\in C_n(\spinct),\\
&\alphas_\spinct=\begin{cases}d_n(\alphas_\spinc), &\text{if }\spinct=\spinc\\
\Upsilon_\spinc(\alpha_\spinc):=\big(0,(d_\B\circ \Xi_\spinc\circ \pi_\spinc)(\x),
(\jmath_\spinc\circ \Xi_\spinc\circ \pi_\spinc)(\x)\big),\ \ \
&\text{if } \spinct=\spinc+n\\
0&\text{Otherwise}
\end{cases}
\end{split}
\end{equation}
It is not hard to see that $\Upsilon_\spinc:C_n(\spinc)\ra C_n(\spinc+n)$ is a chain map. Consequently,
$f_n:F_n[\spinc]\ra F_n[\spinc]$
defines a differential on this vector space.
\begin{thm}\label{thm:surgery-formula}
Suppose that $K\subset X$  is a  knot inside a  homology sphere $X$ of genus $g=g(K)$. For a  class $[\spinc]\in\SpinC(X_n(K))=\Z/n\Z$,
let the complex $(F_n[\spinc],f_n)$  be defined as before.
The Heegaard Floer homology group
$\ov{\mathrm{HF}}(X_n(K),[\spinc])$   may then be computed
as the homology of the chain complex $(F_n[\spinc],f_n)$:
\begin{equation}
\ov{\mathrm{HF}}(X_n(K),[\spinc])=\mathrm{H}_*(F_n[\spinc],f_n)
\end{equation}
\end{thm}
Note that $\pi_\spinc$ is trivial if $|\spinc|>g(K)$, and that the homology of the
chain complex $(C_n(\spinc),d_n)$ is zero unless $-g(K)< \spinc\leq n+g(K)$. This implies that the homology of the
chain complex $(F_n[\spinc],f_n)$ is the same as the homology of the chain complex
\begin{equation}
\begin{split}
&G_n[\spinc]=\bigoplus_{\substack{\spinc\in[\spinc]\subset \Z\\ -g(K)< \spinc\leq n+g(K)}}C_n(\spinc)\\
&g_n:G_n[\spinc]\ra G_n[\spinc],\ \ \ \ g_n=f_n|_{G_n[\spinc]}.
\end{split}
\end{equation}
\begin{cor}\label{cor:compare}
Suppose that $K\subset X$ is a knot inside a homology sphere $X$ and $n>0$ is a given integer. The
 induced knot $K_n\subset X_n(K)$ has simple knot Floer homology if and only if  the maps induced in
 homology by the chain maps
 \begin{equation}
 \Upsilon_\spinc:C_n(\spinc)\lra C_n(\spinc+n)
 \end{equation}
 vanish for all $\spinc\in\RelSpinC(X,K)$.
\end{cor}
\begin{proof}
Immediate corollary of a comparison between the above two surgery formulas.
\end{proof}
\section{Small surgery on a knot}
In this section, we will assume that the surgery coefficient $n$ is small.  \\
\begin{thm}\label{thm:small-surgery}
Suppose that $K\subset X$ is a knot inside a homology sphere $X$ of Seifert genus $g(K)$.
Suppose that $0<n<2g(K)$ is a given integer and $K_n\subset X_n(K)$ is the knot obtained from $K$ by
$n$-surgery. Then $K_n$ can not have simple knot Floer homology, i.e.
\begin{equation}
\rank\big(\ov{\mathrm{HFK}}(X_n(K),K_n)\big)>\rank\big(\ov{\mathrm{HF}}(X_n(K))\big)
\end{equation}
\end{thm}
\begin{proof}
Let us use corollary~\ref{cor:compare} for $\spinc=g(K)$. Note that
$\ov{\mathrm{HFK}}(X_n(K),K_n;{g(K)})$ is isomorphic to the homology of the mapping cone
$$C_n(g(K))=\B\{g\}\oplus \B\{>n-g(K)\}\oplus \B,$$
which is equipped with the differential $d_n$ as before.
Furthermore, 
$$C_n(g(K)+n)=\B\{\geq g(K)+n\}\oplus \B\{>-g(K)\}\oplus \B=0\oplus \B\{>-g(K)\}\oplus \B$$
and its homology is thus isomorphic to $\Hbb\{-g(K)\}=\Hbb\{g(K)\}$.
 The map
$$(\Upsilon_{g(K)})_*:\mathrm{H}_*(C_n(g(K)),d_n)\ra \Hbb\{g(K)\}$$
is thus by the map sending
$$\Hbb\{g(K)\}=\B\{g(K)\}\subset C_n(g(K))=\B\{g(K)\}\oplus\B\{>n-g(K)\}\oplus\B$$ to
$\Hbb\{g(K)\}$ in the target vector space by the identity map.
In order for this map to be trivial, we will need  the map $\Hbb\{g(K)\}\ra \Hbb$ induced by inclusion to be 
injective and its image to be disjoint from the image of the map
$\Hbb\{>n-g(K)\}\ra \Hbb$ induced by inclusion. Let $\x$ be a generator of $\Hbb\{g(K)\}$ which is mapped to a non-trivial
element $[\x]\in\Hbb$. Thus, $\x$ is not in the image of $d_\B\ra d_\B$. If $n<2g(K)$, $n-g(K)\leq g(K)-1$ and
$\B\{>n-g(K)\}$ contains $\x$ as a closed generator. Since $\x$ is not in the image of $d_\B$, it survives in
$\Hbb\{>n-g(K)\}$ and is mapped to the same class $[\x]\in \Hbb$ by $(\imath_{n+1-g(K)})_*$. Thus, $(\x,\x,0)$
is a non-trivial element in $\mathrm{H}_*(C_n(g(K)),d_n)$ which is mapped to $\x\in\Hbb\{g(K)\}$ by
$(\Upsilon_{g(K)})_*$. This contradiction proves the theorem.
\end{proof}
\section{Large surgery on a knot}
Suppose now that $n\geq2g(K)$. As before, all the maps $(\Upsilon_\spinc)_*$ should vanish. If the relative
$\SpinC$ structure $\spinc\in\RelSpinC(X,K)=\Z$ satisfies $-g(K)<\spinc\leq g(K)$, we would have
$n-\spinc\geq g(K)$ and $n+\spinc>g(K)$, which imply that $\Hbb\{\geq \spinc+n\}$ and $\Hbb\{>n-\spinc\}$
are trivial. Thus $C_n(\spinc)=\Hbb\{\geq \spinc\}\oplus0\oplus \Hbb$ and
$C_n(\spinc+n)=0\oplus \Hbb\{>-\spinc\}\oplus \Hbb$. This implies that
\begin{equation}
\begin{split}
\mathrm{H}_*(C_n(\spinc),d_n)&=\Hbb\{<\spinc\},\ \ \& \\
\mathrm{H}_*(C_n(\spinc+n),d_n)&=\Hbb\{\leq \spinc\}. \\
\end{split}
\end{equation}
Under the above identifications the map $(\Upsilon_\spinc)_*$
which will be be denoted by $\epsilon_\spinc:\Hbb\{<\spinc\}\ra \Hbb\{\leq -\spinc\}$
may be described as follows. It is the map  obtained by first taking
$\Hbb\{<\spinc\}$ to $\Hbb\{\spinc\}$ using the map $\tau_\spinc$ induced by the differential
$d_\B$ of the complex $\B$, then using the duality map $\taub_\spinc$
to take $\Hbb\{\spinc\}$ to $\Hbb\{-\spinc\}$, and finally going from $\Hbb\{-\spinc\}$ 
to $\Hbb\{\leq -\spinc\}$ using
the map $q_{-\spin}$ induced in homology by the inclusion of the first vector space in the quotient complex
$\B\{\leq -\spinc\}$ of $\B$. The assumption that $K_n\subset X_n(K)$ has simple knot Floer homology
then implies that all the maps $\epsilon_\spinc$ should vanish.\\

We may now prove the following theorem, which should be compared with theorem~2.4 from \cite{Matt}.
\begin{thm}\label{thm:large-surgery}
Let $K\subset X$ be a knot in a homology sphere $L$-space $X$ of Seifert genus $g(K)$, and fix the 
above notation.
If for all $-g(K)<\spinc\leq g(K)$ the maps $\epsilon_\spinc$ vanish, there is an increasing sequence of integers
$$-g(K)=n_{-k}<n_{1-k}<...<n_k=g(K)$$
with $n_i=-n_{-i}$ for which the following is true. For $i\in \Z$ with $|i|\leq g(K)$ define
$$\delta_i=\begin{cases}
0+d(Y) &\text{if }i=g(K)\\
\delta_{i+1}+2(n_{i+1}-n_i)-1 &\text{if } i<g(K),\ \&\ g(K)-i\equiv 1\ (\text{mod }2)\\
\delta_{i+1}+1 &\text{if } i<g(K),\ \&\ g(K)-i\equiv 0\ (\text{mod }2).\\
\end{cases} $$
In this situation, $\ov{\mathrm{HFK}}(K,\spinc)=0$ unless $\spinc=n_i$ for some $-k\leq i\leq k$, in which case
$\ov{\mathrm{HFK}}(K,\spinc)=\F$ and it is supported entirely in homological degree $\delta_i$.
 If the generator of $\ov{\mathrm{HFK}}(K,n_i)$
is denoted by $\x_i$, the filtered chain complex $(\B(K),d_\B)$ may be described as 
$\B=\langle \x_{-k},...,\x_k\rangle$, where the differential is given by
$$d_\B(\x_i)=\begin{cases}
0\ &\text{if }i=k,\ \text{or } 0<k-i \text{ is odd,}\\
\x_{i+1}\ &\text{otherwise.}
\end{cases}$$
\end{thm}
\begin{proof}
Let us assume that
$$-g(K)=n_{-k}<n_{1-k}<...<n_k=g(K)$$
are the values for $\spinc$ so that the associated vector space $\Hbb\{\spinc\}$ is non-trivial. From the
duality isomorphism of knot Floer homology groups we thus know that $n_{-i}=-n_i$ for each $i$.
Let us denote the generators of $\Hbb\{n_i\}$ by $\y_i^1,\y_i^2,...,\y_i^{\ell_i}$.
We may assume that under the duality map $\taub(\y_{-i}^j)=\y_i^j$.
From the argument of our previous section, we know that $\ell_k=1$, and we may thus denote $\y_k^1$ by $\x_k$ and
$\y_{-k}^1$ by $\x_{-k}$. We also know that $\x_k$ survives in the homology of $\B(K)$, i.e. it is not in the
image of $d_\B$.\\
We will prove the following claim by an induction on $i$, which implies the above theorem.\\
{\bf{Claim.}} {\emph{In the above situation, $\ell_k=\ell_{k-1}=...=\ell_{k-i}=1$. Furthermore, we may choose
$\x_{k-j}\in \B\{\geq n_{k-j}\}$ and $\x_{j-k}\in\B\{\geq n_{j-k}\}$ so that they  survive in the homology of
quotient complexes $\Hbb\{n_{k-j}\}=\B\{\geq n_{k-j}\}/\B\{> n_{k-j}\}$ and
$\Hbb\{n_{j-k}\}=\B\{\geq n_{j-k}\}/\B\{> n_{j-k}\}$ respectively, so that
\begin{equation}\begin{split}
d_\B(\x_{k-j})&=\begin{cases}
0\ &\text{if }j=0,\ \text{or } j \text{ is odd,}\\
\x_{k-j+1}\ &\text{otherwise.}
\end{cases},\ \ j=0,1,...,i\\
 d_\B(\x_{j-k})&=\begin{cases}
0\ &\text{if } j \text{ is odd,}\\
\x_{j-k+1}\ &\text{otherwise.}
\end{cases},\ \ \ \ \ \ \ \ \  j=0,1,...,i-1
\end{split}\end{equation}
}}\\
From the above considerations, the case $i=0$ is already proved. Now, assume that the claim is true for $i$. We will
prove that it will also follow for $i+1$. We will need to consider two cases depending
on the parity of $i$.\\
First, assume that $i$ is even. For any value of $\spinc$, we have 
$\epsilon_\spinc=q_{-\spinc}\circ \taub_\spinc\circ \tau_\spinc=0$. 
If the map $\tau_\spinc:\Hbb\{<\spinc\}\ra \Hbb\{\spinc\}$ is
not an isomorphism, the homology group $\Hbb\{\leq \spinc\}$ would be non-trivial.
However, the differential $d_\B$ of the complex $\B$ induces a map $p_\spinc:\Hbb\{\leq \spinc\}\ra
\Hbb\{>\spinc\}$ and its image does not cover the class of $\x_k$ in the target. Since the homology of the 
mapping cone of $p_\spinc$ is $\Hbb$, this means that $p_\spinc$ is injective and that the rank of 
$\Hbb\{\leq \spinc\}$ is one less that the rank of $\Hbb\{>\spinc\}$.   
If we set $\spinc=n_{k-i-1}$, since there are no generators associated with relative $\SpinC$ classes
$n_{k-i-1}<\spinct<n_{k-i}$, we have $\Hbb\{>\spinc\}=\Hbb\{\geq n_{k-i}\}=\F$ by induction hypothesis.
The above observation thus implies that $\Hbb\{\leq n_{k-i-1}\}$ is trivial. Thus the map $\tau_{n_{k-i-1}}$
is an isomorphism. Since for $\spinc=n_{k-i-1}$, $\tau_\spinc$ and $\taub_{\spinc}$ are isomorphisms and 
$\epsilon_\spinc$ is trivial, we may conclude that $q_{-\spinc}=0$. This last assumption is equivalent to 
the assumption that $\tau_{-\spinc}:\Hbb\{<-\spinc\}\ra \Hbb\{-\spinc\}$ is surjective.
However, for $\spinc=n_{k-i-1}$, setting $j=i+1-k$ we have
$$\Hbb\{<-\spinc\}=\Hbb\{\leq n_{i-k}\}=\langle\x_{i-k}\rangle
,\ \&\ \ \Hbb\{-\spinc\}=\langle \y^{1}_{j},...,\y^{\ell_{j}}_j\rangle.$$
Surjectivity of $\tau_{-\spinc}$ implies that $\ell_j=1$ and that $\tau_{-\spinc}$ is an isomorphism. 
Setting $\x_{i+1-k}=\y^1_{i+1-k}$, this means that $d_\B(\x_{i-k})$ is equal to $\x_{i+1-k}$ plus 
terms in higher filtration levels. 
After a suitable change of basis for the filtered chain complex we may assume that $d_\B(\x_{i-k})= \x_{i+1-k}$.
On the other hand, $\Hbb\{n_{k-i-1}\}=\Hbb\{n_{i+1-k}\}=\langle\x_{i+1-k} \rangle$ is generated by a 
generator $\x_{k-i-1}=\taub(\x_{i+1-k})$. Since the homology group $\Hbb\{\geq n_{k-i}\}=\Hbb\{>n_{k-i-1}\}$
is generated by $\x_k$ and the image of $\x_{k-i-1}$ under $d_\B$ in $\Hbb\{>n_{k-i-1}\}$ 
can not be equal to $\x_k$, 
after another suitable change of basis for the filtered chain complex $\B$
we may assume that $d_\B(\x_{k-i-1})=0$. This completes the proof of the assertions of the above claim 
for $i+1$ when $i$ is even.\\
If $i$ is odd, $\Hbb\{\leq n_{i-k}\}$ is trivial by induction hypothesis. This implies that $\Hbb\{\leq n_{i+1-k}\}$
is isomorphic to $\Hbb\{n_{i+1-k}\}$ and the isomorphism is given by the map 
$$q_{n_{i+1-k}}:\Hbb\{n_{i+1-k}\}\lra \Hbb\{\leq n_{i+1-k}\}.$$
Since $\epsilon_{n_{k-i-1}}$ is trivial, we may conclude that $\tau_{n_{k-i-1}}$ is trivial. 
In other words, $\Hbb\{n_{k-i-1}\}$ is not in the image of $d_\B$. As a result, the map $r_{n_{k-i-1}}$
from $\Hbb\{n_{k-i-1}\}$ to $\Hbb\{>n_{k-i-1}\}=\Hbb\{\geq n_{k-i}\}$, 
induced by the differential $d_\B$,  is injective. From  induction
hypothesis we may conclude that $\Hbb\{\geq n_{k-i}\}=\langle \x_{k-i},\x_k\rangle$, and $\x_k$ is 
not in the image of $d_\B$. Thus $\Hbb\{n_{k-i-1}\}$ is generated by a single generator 
$\x_{k-i-1}=\y^{1}_{k-i-1}$ and we may assume that $d_\B(\x_{k-i-1})=\x_{k-i}$. Since 
$\Hbb\{<n_{i+1-k}\}=\Hbb\{\leq n_{i-k}\}$ is trivial, $\x_{i+1-k}=\taub(\x_{k-i-1})$ is not in the image 
of $d_\B$ and the assertions of the above claim are thus satisfied for $i+1$ if $i$ is odd.
This completes the induction and proves the above claim.\\ 
\\
In order to complete the proof of the theorem, let us denote the homological degree of the generator 
$\x_j$ by $\delta_j=\delta(\x_j)$. Since $\Hbb$ is generated by $\x_k$, it is clear that $\delta_k=d(X)$ and since 
for even values of $i$, $d_\B(\x_{k-i})=\x_{k-i+1}$, we have $\delta_{k-i}=\delta_{k-i+1}+1$.
Also, note that $\delta_{-j}=2n_j+\delta_j$. Thus for any odd value of $i$ we have
\begin{equation}\begin{split}
\delta_{k-i}=-2n_{k-i}+\delta_{i-k}&=-2n_{k-i}+\delta(d_\B(\x_{i-k-1}))\\ 
&=-2n_{k-i}+(\delta_{i-k-1}-1)\\
&=-2n_{k-i}+2n_{k-i+1}+\delta_{k-i+1}-1.
\end{split}\end{equation}
This completes the proof of the theorem.
\end{proof}
We are almost done with the proof of our main theorem in this paper (theorem~\ref{thm:main}).\\
\begin{proof}(of theorem~\ref{thm:main})
If $n\geq 2g(K)$ and $K_n$ has simple knot Floer homology, 
we have seen that the hypothesis of theorem~\ref{thm:large-surgery} are satisfied. Thus 
the filtered chain complex $(\B(K),d_\B)$ has the special form 
described in theorem~\ref{thm:large-surgery}. From here, it is clear that all the maps
$\epsilon_\spinc:\Hbb\{<\spinc\}\ra\Hbb\{\leq -\spinc\}$ vanish. In fact, for any given 
relative $\SpinC$ structure $\spinc$ satisfying $-g(K)<\spinc\leq g(K)$ precisely one of the 
two groups $\Hbb\{<\spinc\}$ and  $\Hbb\{\leq -\spinc\}$ is isomorphic to 
$\Z/2\Z$ and the other one is trivial. Thus $\ov{\mathrm{HF}}(X_n(K),[\spinc])=\Z/2\Z$ by our 
previous considerations. The isomorphism  $\ov{\mathrm{HF}}(X_n(K),[\spinc])=\ov{\mathrm{HF}}(X)=\Z/2\Z$
is clear for other $\SpinC$ structures. Thus, $X_n(K)$ is a $L$-space and one direction of the 
theorem is proved.\\
For the other direction, suppose that $(\B(K),d_\B)$ has the structure described in theorem~\ref{thm:large-surgery}.
Theorem~\ref{thm:knot-surgery-formula} then gives a description of the Floer homology 
groups $\ov{\mathrm{HFK}}(X_n(K),K_n;\spinc)$
for different values of $\spinc\in\RelSpinC(X,K)=\Z$ as the homology of the mapping cone of a map 
\begin{equation}
h_n:\Hbb\{\geq \spinc\}\oplus\Hbb\{> n-\spinc\}\ra \Hbb=\Z/2\Z.
\end{equation} 
If $\spinc\leq g(K)$, $n-\spinc\geq n-g(K)\geq g(K)$ and $\Hbb\{>n-g(K)\}=0$. In this 
situation $\ov{\mathrm{HFK}}(X_n(K),K_n;\spinc)
=\Hbb\{<\spinc\}$ and 
$$\Hbb\{<\spinc\}=\begin{cases}0\  &\text{if } n_{k-i-1}<\spinc\leq n_{k-i},\ \&\ i\ \text{is even}\\
\Z/2\Z\  \ &\text{if } n_{k-i-1}<\spinc\leq n_{k-i},\ \&\ i\ \text{is odd}\\
0 &\text{if } \spinc\leq n_{-k}=-g(K)\\
\end{cases}.$$ 
On the other hand,  a similar argument shows that for $s>g(K)$ we have $\ov{\mathrm{HFK}}(X_n(K),K_n;\spinc)
=\Hbb\{<\spinc\}$ and 
$$\Hbb\{\leq n-\spinc\}=\begin{cases}0\  &\text{if } n-n_{k-i}<\spinc\leq n- n_{k-i-1},\ \&\ i\ \text{is even}\\
\Z/2\Z\  \ &\text{if } n-n_{k-i}<\spinc\leq n-n_{k-i-1},\ \&\ i\ \text{is odd}\\
0 &\text{if } \spinc> n-n_{-k}=n+g(K)\\
\end{cases}.$$
For a relative $\SpinC$ structure $\spinc\in\SpinC(X,K)=\Z$ the non-triviality assumption 
$\ov{\mathrm{HFK}}(X_n(K),K_n;\spinc)\neq 0$
 thus implies that $-g(K)<\spinc\leq n+g(K)$.
If for two relative $\SpinC$ structures $\spinc$ and $\spinct$ both $\ov{\mathrm{HFK}}(X_n(K),K_n;\spinc)$ and 
$\ov{\mathrm{HFK}}(X_n(K),K_n;\spinct)$ are non-trivial and $\spinct-\spinc$ 
is a positive multiple of $n$, we should have $-g(K)<\spinc\leq g(K)$ and $\spinct=n+\spinc$.
Furthermore, $n_{k-i-1}<\spinc\leq n_{k-i}$ for an even integer $i$. Thus
\begin{equation}
\begin{split}
n-n_{k-(2k-i-1)}=n+n_{k-i-1}<&\spinct=\spinc+n\\ 
&\leq n+n_{k-i}=n-n_{k-(2k-i-1)-1},
\end{split}
\end{equation} 
where $j=2k-i-1$ is an odd integer. 
From the above description we have $\ov{\mathrm{HFK}}(X_n(K),K_n;\spinct)=0$, which is 
a contradiction. Thus from all
 relative $\SpinC$ classes in $\RelSpinC(X,K)$ in the congruence class of 
$\spinc$ modulo $n$, at most one of them, say $\spinct$ has the property that
$\ov{\mathrm{HFK}}(X_n(K),K_n;\spinct)\neq 0$. This clearly implies that
$K_n$ has simple knot Floer homology, completing the proof.
\end{proof}

\end{document}